\documentclass{article}
\usepackage{amsfonts}
\usepackage{amsmath}
\usepackage{hyperref}

\newtheorem{theorem}{Theorem}

\newtheorem{corollary}[theorem]{Corollary}

\newtheorem{definition}[theorem]{Definition}

\newtheorem{lemma}[theorem]{Lemma}

\newtheorem{remark}[theorem]{Remark}

\newenvironment{proof}[1][Proof]{\noindent\textbf{#1.} }{\ \rule{0.5em}{0.5em}}
\begin{document}

\title{Sums involving the digamma function connected to the incomplete beta
function and the Bessel functions}
\author{Juan L. Gonz\'{a}lez-Santander, Fernando S\'{a}nchez Lasheras \\
Department of Mathematics, Universidad de Oviedo,\\
33007 Oviedo, Spain.}
\maketitle

\begin{abstract}
We calculate some infinite sums containing the digamma function in
closed-form. These sums are related either to the incomplete beta function
or to the Bessel functions. The calculations yield interesting new results
as by-products, such as parameter differentiation formulas for the beta
incomplete function, reduction formulas of $_{3}F_{2}$ hypergeometric
functions, or a definite integral which does not seem to be tabulated in the
most common literature. As an application of some sums involving the digamma
function, we have calculated some redution formulas for the parameter
differentiation of the Mittag-Leffler function and the Wright function.
\end{abstract}

\section{Introduction}

In the literature \cite{Hansen}, \cite{Brychov}, we found some compilations
of series and finite sums involving the digamma function. Some authors have
contributed to these compilations, such as Doelder \cite{Doelder}, Miller
\cite{Miller}, and Cvijovi\'{c} \cite{Cvijovic}. More recently, the authors
have published some novel results in this regard \cite{SumsPsiJL}.

Sums involving the digamma function occur in the expressions of the
derivatives of the Mittag-Leffler function and the Wright function with
respect to parameters \cite{Apelblat1,Apelblat2}. Also, they occur in the
derivation of asymptotic expansions for Mellin-Barnes integrals \cite{Paris}.

The aim of this paper is the derivation of some new sums involving the
digamma function by using the derivative of the Pochhammer symbol to known
reduction formulas of the generalized hypergeometric function. As a
consitency test, for many particular values of the results obtained, we
recover expressions given in the literature. In adittion, we have developed
a MATHEMATICA\ program to numerically check all the new expressions derived
in the paper. This program is available at
\url{https://bit.ly/3LG2gej}%
.

This paper is organized as follows. In Section \ref{Section: Preliminaries},
we present some basic properties of the Pochhammer symbol, the beta and the
digamma functions, as well as the definitions of the generalized
hypergeometric function and the Meijer-$G$ function. In Section \ref%
{Section: Incomplete Beta}, we derive some sums connected to the parameter
differentiation of the incomplete beta function. In Section \ref{Section:
Bessel}, we calculate in a similar way some other sums connected to the
order derivatives of the Bessel and the modified Bessel functions. Section %
\ref{Section: Application}\ is devoted to the application of some sums
involving the digamma function to the calculation of reduction formulas of
the parameter differentiation of the Wright and Mittag-Leffler functions.
Finally, we collect our conclusions in Section \ref{Section: Conclusions}.

\section{Preliminaries\label{Section: Preliminaries}}

The Pochhamer symbol is defined as \cite[Eqn. 18:12:1]{Atlas}%
\begin{equation}
\left( x\right) _{n}=\frac{\Gamma \left( x+n\right) }{\Gamma \left( x\right)
},  \label{Pochhammer_def}
\end{equation}%
where $\Gamma \left( x\right) $ denotes the gamma function with the
following basic properties \cite[Ch. 43]{Atlas}:\
\begin{eqnarray}
\Gamma \left( z+1\right)  &=&z\,\Gamma \left( z\right) ,
\label{Gamma_factorial} \\
2^{2z-1}\Gamma \left( z\right) \Gamma \left( z+\frac{1}{2}\right)  &=&\sqrt{%
\pi }\,\Gamma \left( 2z\right) .  \label{Gamma_duplication}
\end{eqnarray}

Also, the beta function, defined as \cite[Eqn. 1.5.3]{Lebedev}%
\begin{eqnarray*}
&&\mathrm{B}\left( x,y\right) =\int_{0}^{1}t^{x-1}\left( 1-t\right) ^{y-1}dt,
\\
&&\mathrm{Re}\,x,\,\mathrm{Re}\,y>0,
\end{eqnarray*}%
satisfies the property \cite[Eqn. 1.5.5]{Lebedev}
\begin{equation}
\mathrm{B}\left( x,y\right) =\frac{\Gamma \left( x\right) \Gamma \left(
y\right) }{\Gamma \left( x+y\right) }.  \label{Beta_property}
\end{equation}

Further, the incomplete beta function is defined as \cite[Eqn. 8.17.1]{NIST}:%
\begin{equation}
\mathrm{B}_{z}\left( x,y\right) =\int_{0}^{z}t^{x-1}\left( 1-t\right)
^{y-1}dt,  \label{Beta_incomplete_def}
\end{equation}%
which satisfies the property\ \cite[Eqn. 58:5:1]{Atlas},%
\begin{equation}
\mathrm{B}_{z}\left( a,b\right) +\mathrm{B}_{1-z}\left( b,a\right) =\mathrm{B%
}\left( a,b\right) .  \label{Beta_reflection}
\end{equation}

A function related to the incomplete beta function is the Lerch function,
defined as
\begin{equation}
\Phi \left( z,a,b\right) =\sum_{k=0}^{\infty }\frac{z^{k}}{\left( k+b\right)
^{a}}.  \label{Lerch_def}
\end{equation}

According to (\ref{Pochhammer_def}), we have%
\begin{equation}
\frac{d}{dx}\left[ \left( x\right) _{n}\right] =\left( x\right) _{n}\left[
\psi \left( x+n\right) -\psi \left( x\right) \right] ,  \label{D[(x)_n]}
\end{equation}%
and
\begin{equation}
\frac{d}{dx}\left[ \frac{1}{\left( x\right) _{n}}\right] =\frac{1}{\left(
x\right) _{n}}\left[ \psi \left( x\right) -\psi \left( x+n\right) \right] ,
\label{D[1/(x)_n]}
\end{equation}%
where $\psi \left( x\right) $ denotes the digamma function \cite[Ch. 44]%
{Atlas}%
\begin{equation*}
\psi \left( x\right) =\frac{\Gamma ^{\prime }\left( x\right) }{\Gamma \left(
x\right) },
\end{equation*}%
with the following properties \cite[Eqns. 1.3.3-4\&8]{Lebedev}
\begin{eqnarray}
\psi \left( \frac{1}{2}\right)  &=&-\gamma -2\ln 2,  \label{psi(1/2)} \\
\psi \left( z+1\right)  &=&\frac{1}{z}+\psi \left( z\right) ,
\label{psi(1+z)} \\
\psi \left( 1-z\right) -\psi \left( z\right)  &=&\pi \cot \left( \pi
z\right) .  \label{psi(1-z)-psi(z)}
\end{eqnarray}

Finally, $_{p}F_{q}\left( z\right) $ denotes the generalized hypergeometric
function, usually defined by means of the hypergeometric series \cite[Sect.
16.2]{NIST}:%
\begin{equation}
_{p}F_{q}\left( \left.
\begin{array}{c}
a_{1},\ldots ,a_{p} \\
b_{1},\ldots b_{q}%
\end{array}%
\right\vert z\right) =\sum_{k=0}^{\infty }\frac{\left( a_{1}\right)
_{k}\cdots \left( a_{p}\right) _{k}}{\left( b_{1}\right) _{k}\cdots \left(
b_{q}\right) _{k}}\frac{z^{k}}{k!},  \label{pFq_def}
\end{equation}%
whenever this series converge and elsewhere by analytic continuation.

In addition, the Meijer-$G$ function is defined via Mellin-Barnes integral
representation \cite[Eqn. 16.17.1]{NIST}:%
\begin{eqnarray*}
&&G_{p,q}^{m,n}\left( z\left\vert
\begin{array}{c}
a_{1},\ldots ,a_{p} \\
b_{1},\ldots b_{q}%
\end{array}%
\right. \right) \\
&=&\frac{1}{2\pi i}\int_{L}\frac{\prod_{\ell =1}^{m}\Gamma \left( b_{\ell
}-s\right) \prod_{\ell =1}^{m}\Gamma \left( 1-a_{\ell }+s\right) }{%
\prod_{\ell =m}^{q-1}\Gamma \left( 1-b_{\ell +1}+s\right) \prod_{\ell
=n}^{p-1}\Gamma \left( a_{\ell +1}-s\right) }\,z^{s}ds,
\end{eqnarray*}%
where the integration path $L$ separates the poles of the factors $\Gamma
\left( b_{\ell }-s\right) $ from those of the factors $\Gamma \left(
1-a_{\ell }+s\right) $.

\section{Sums connected to the incomplete beta function\label{Section:
Incomplete Beta}}

\subsection{Derivatives of the incomplete beta function with respect to the
parameters}

\begin{theorem}
The following parameter derivative holds true:%
\begin{equation}
\frac{\partial }{\partial a}\mathrm{B}_{z}\left( a,b\right) =\ln z\,\mathrm{B%
}_{z}\left( a,b\right) -\frac{z^{a}}{a^{2}}\,_{3}F_{2}\left( \left.
\begin{array}{c}
1-b,a,a \\
a+1,a+1%
\end{array}%
\right\vert z\right) .  \label{DaB_resultado}
\end{equation}
\end{theorem}

\begin{proof}
According to the definition of the incomplete beta function (\ref%
{Beta_incomplete_def}), we have%
\begin{equation}
\frac{\partial }{\partial a}\mathrm{B}_{z}\left( a,b\right)
=\int_{0}^{z}t^{a-1}\left( 1-t\right) ^{b-1}\ln t\,dt.  \label{DaB_int}
\end{equation}%
Now, apply the formulas \cite[Eqn. 18:3:4]{Atlas},%
\begin{equation}
\frac{1}{\left( 1-t\right) ^{\nu }}=\sum_{k=0}^{\infty }\left( \nu \right)
_{k}\frac{t^{k}}{k!},  \label{Binomial_theorem}
\end{equation}%
and \cite[Eqn. 1.6.1(18)]{Prudnikov1}
\begin{equation}
\int x^{p}\ln x\,dx=x^{p+1}\left[ \frac{\ln x}{p+1}-\frac{1}{\left(
p+1\right) ^{2}}\right] ,  \label{Int_xp_lnx}
\end{equation}%
in order to rewrite (\ref{DaB_int}) as
\begin{equation}
\frac{\partial }{\partial a}\mathrm{B}_{z}\left( a,b\right) =z^{a}\left\{
\ln z\sum_{k=0}^{\infty }\frac{\left( 1-b\right) _{k}\,z^{z}}{k!\left(
a+k\right) }-\sum_{k=0}^{\infty }\frac{\left( 1-b\right) _{k}}{k!\left(
a+k\right) ^{2}}\right\} .  \label{DaBz_sums}
\end{equation}%
Taking into account the property%
\begin{equation*}
\frac{1}{\alpha +k}=\frac{\left( \alpha \right) _{k}}{\alpha \left( \alpha
+1\right) _{k}},
\end{equation*}%
and the definition of the generalized hypergeometric function (\ref{pFq_def}%
), we may recast the the sums given in (\ref{DaBz_sums})\ as%
\begin{equation}
\frac{\partial }{\partial a}\mathrm{B}_{z}\left( a,b\right) =z^{a}\left\{
\frac{\ln z}{a}\,_{2}F_{1}\left( \left.
\begin{array}{c}
1-b,a \\
a+1%
\end{array}%
\right\vert z\right) -\frac{1}{a^{2}}\,_{3}F_{2}\left( \left.
\begin{array}{c}
1-b,a,a \\
a+1,a+1%
\end{array}%
\right\vert z\right) \right\} .  \label{DaB_hyper}
\end{equation}%
Finally, apply to (\ref{DaB_hyper})\ the reduction formula \cite[Eqn.
7.3.1(28)]{Prudnikov3}%
\begin{equation}
_{2}F_{1}\left( \left.
\begin{array}{c}
\alpha ,\beta \\
\beta +1%
\end{array}%
\right\vert z\right) =\beta \,z^{-\beta }\,\mathrm{B}_{z}\left( \beta
,1-\alpha \right) ,  \label{2F1_reduction_beta_2}
\end{equation}%
in order to arrive at (\ref{DaB_resultado}), as we wanted to prove.
\end{proof}

As a consequence of the last theorem, we calculate the next integral, which
does not seem to be tabulated in the most common literature.

\begin{theorem}
For $\mathrm{Re}\,\alpha >-1$ and $\mathrm{Re}\,z>0$, the following integral
holds true:\
\begin{equation}
\int_{0}^{z}\frac{u^{\alpha }}{1-u^{2}}\ln u\,du=\frac{1}{2}\ln z\,\mathrm{B}%
_{z^{2}}\left( \frac{1+\alpha }{2},0\right) -\frac{z^{\alpha +1}}{4}\,\Phi
\left( z^{2},2,\frac{1+\alpha }{2}\right) .  \label{Integral_Lerch_resultado}
\end{equation}
\end{theorem}

\begin{proof}
According to \cite[Eqn. 58:14:7]{Atlas}, we have%
\begin{equation*}
\mathrm{B}_{\tanh ^{2}\left( x\right) }\left( \lambda ,0\right)
=2\int_{0}^{x}\tanh ^{2\lambda -1}t\,dt.
\end{equation*}%
Performing the substitutions $z=\tanh x$ and $u=\tanh t$, we obtain,%
\begin{equation}
\mathrm{B}_{z^{2}}\left( \lambda ,0\right) =2\int_{0}^{z}\frac{u^{2\lambda
-1}}{1-u^{2}}\,du.  \label{Beta(z^2,lambda,0)}
\end{equation}%
On the one hand, calculate the derivative of the LHS\ of (\ref%
{Beta(z^2,lambda,0)})\ with respect to the parameter $\lambda $ taking into
account (\ref{DaB_resultado}),%
\begin{equation}
\frac{\partial }{\partial \lambda }\mathrm{B}_{z^{2}}\left( \lambda
,0\right) =2\ln z\,\mathrm{B}_{z^{2}}\left( \lambda ,0\right) -\frac{%
z^{2\lambda }}{\lambda ^{2}}\,_{3}F_{2}\left( \left.
\begin{array}{c}
1,\lambda ,\lambda  \\
\lambda +1,\lambda +1%
\end{array}%
\right\vert z^{2}\right) .  \label{D[Beta(z^2,lambda,0)]}
\end{equation}%
In order to calculate the $_{3}F_{2}$ function given in (\ref%
{D[Beta(z^2,lambda,0)]}), we apply the reduction formula \cite[Eqn. 7.4.1(5)]%
{Prudnikov3}%
\begin{equation*}
_{3}F_{2}\left( \left.
\begin{array}{c}
a,b,c \\
a+1,b+1%
\end{array}%
\right\vert x\right) =\frac{1}{b-a}\left[ b\,_{2}F_{1}\left( \left.
\begin{array}{c}
a,c \\
a+1%
\end{array}%
\right\vert x\right) -a\,_{2}F_{1}\left( \left.
\begin{array}{c}
b,c \\
b+1%
\end{array}%
\right\vert x\right) \right] ,
\end{equation*}%
thus, taking $c=1$, and applying the reduction formula \cite[Eqn. 7.3.1(122)]%
{Prudnikov3}%
\begin{equation*}
_{2}F_{1}\left( \left.
\begin{array}{c}
1,b \\
b+1%
\end{array}%
\right\vert x\right) =b\,\,\Phi \left( x,1,b\right) ,
\end{equation*}%
as well as the definition of the Lerch function (\ref{Lerch_def}), we have%
\begin{eqnarray*}
_{3}F_{2}\left( \left.
\begin{array}{c}
1,a,b \\
a+1,b+1%
\end{array}%
\right\vert x\right)  &=&\frac{ab}{b-a}\left[ \Phi \left( x,1,a\right) -\Phi
\left( x,1,b\right) \right]  \\
&=&\frac{ab}{b-a}\sum_{k=0}^{\infty }x^{k}\left( \frac{1}{k+a}-\frac{1}{k+b}%
\right)  \\
&=&ab\sum_{k=0}^{\infty }\frac{x^{k}}{\left( k+a\right) \left( k+b\right) }.
\end{eqnarray*}%
Therefore, taking $a=\lambda $, $b=\lambda +\epsilon $, we have%
\begin{eqnarray}
_{3}F_{2}\left( \left.
\begin{array}{c}
1,\lambda ,\lambda  \\
\lambda +1,\lambda +1%
\end{array}%
\right\vert z^{2}\right)  &=&\lim_{\epsilon \rightarrow 0}\lambda \left(
\lambda +\epsilon \right) \sum_{k=0}^{\infty }\frac{z^{2k}}{\left( k+\lambda
\right) \left( k+\lambda +\epsilon \right) }  \notag \\
&=&\lambda ^{2}\,\Phi \left( z^{2},2,\lambda \right) .
\label{3F2_Lerch_resultado}
\end{eqnarray}%
Insert (\ref{3F2_Lerch_resultado})\ in (\ref{D[Beta(z^2,lambda,0)]})\ to
obtain%
\begin{equation}
\frac{\partial }{\partial \lambda }\mathrm{B}_{z^{2}}\left( \lambda
,0\right) =2\ln z\,\mathrm{B}_{z^{2}}\left( \lambda ,0\right) -z^{2\lambda
}\,\Phi \left( z^{2},2,\lambda \right) .  \label{LHS_beta}
\end{equation}%
On the other hand, calculate the derivative of the RHS\ of (\ref%
{Beta(z^2,lambda,0)}) as
\begin{equation}
2\frac{\partial }{\partial \lambda }\int_{0}^{z}\frac{u^{2\lambda -1}}{%
1-u^{2}}\,du=4\int_{0}^{z}\frac{u^{2\lambda -1}}{1-u^{2}}\,\ln u\,du.
\label{RHS_beta}
\end{equation}%
Finally, equate (\ref{LHS_beta})\ to (\ref{RHS_beta})\ and perform the
substitution $\alpha =2\lambda -1$ to complete the proof.
\end{proof}

\begin{lemma}
For $\,\alpha \neq 1,\,\mathrm{Re}\,\alpha <2$, the following reduction
formula holds true:%
\begin{equation}
_{3}F_{2}\left( \left.
\begin{array}{c}
\alpha ,\beta ,\beta \\
\beta +1,\beta +1%
\end{array}%
\right\vert 1\right) =\beta ^{2}\,\mathrm{B}\left( 1-\alpha ,b\right) \left[
\psi \left( 1+\beta -\alpha \right) -\psi \left( \beta \right) \right] .
\label{3F2(1)}
\end{equation}%
Take $a=\beta +\epsilon $, $b=\beta $, $c=\alpha $, and calculate the limit $%
\epsilon \rightarrow 0$ in the following reduction formula \cite[Eqn.
7.4.4(16)]{Prudnikov3},
\begin{eqnarray*}
&&_{3}F_{2}\left( \left.
\begin{array}{c}
a,b,c \\
a+1,b+1%
\end{array}%
\right\vert 1\right) =\frac{a\,b}{a-b}\Gamma \left( 1-c\right) \left\{ \frac{%
\Gamma \left( b\right) }{\Gamma \left( 1+b-c\right) }-\frac{\Gamma \left(
a\right) }{\Gamma \left( 1+a-c\right) }\right\} , \\
&&a\neq b,\,c\neq 1,\,\mathrm{Re}\,c<2,
\end{eqnarray*}%
to obtain:%
\begin{eqnarray}
&&_{3}F_{2}\left( \left.
\begin{array}{c}
\alpha ,\beta ,\beta \\
\beta +1,\beta +1%
\end{array}%
\right\vert 1\right)  \label{3F2_limit} \\
&=&\lim_{\epsilon \rightarrow 0}\beta \left( \beta +\epsilon \right)
\,\Gamma \left( 1-\alpha \right) \,\frac{\Gamma \left( \beta \right)
\,\Gamma \left( 1+\beta -\alpha +\epsilon \right) -\Gamma \left( 1+\beta
-\alpha \right) \,\Gamma \left( \beta +\epsilon \right) }{\epsilon \,\Gamma
\left( 1+\beta -\alpha \right) \,\Gamma \left( 1+\beta -\alpha +\epsilon
\right) }.  \notag
\end{eqnarray}%
Apply the Taylor series expansion
\begin{equation*}
\Gamma \left( x+\epsilon \right) =\Gamma \left( x\right) +\Gamma \left(
x\right) \,\psi \left( x\right) \,\epsilon +O\left( \epsilon ^{2}\right) ,
\end{equation*}%
to calculate (\ref{3F2_limit}). After simplification, we arrive at (\ref%
{3F2(1)}), as we wanted to prove.
\end{lemma}

\begin{theorem}
The following parameter derivative holds true:%
\begin{eqnarray}
\frac{\partial }{\partial b}\mathrm{B}_{z}\left( a,b\right) &=&\frac{\left(
1-z\right) ^{b}}{b^{2}}\,_{3}F_{2}\left( \left.
\begin{array}{c}
1-a,b,b \\
b+1,b+1%
\end{array}%
\right\vert 1-z\right)  \label{DbB_resultado} \\
&&-\ln \left( 1-z\right) \,\mathrm{B}_{1-z}\left( b,a\right) -\mathrm{B}%
\left( a,b\right) \left[ \psi \left( a+b\right) -\psi \left( b\right) \right]
.  \notag
\end{eqnarray}
\end{theorem}

\begin{proof}
According to the definition of the incomplete beta function (\ref%
{Beta_incomplete_def}), we have
\begin{equation*}
\frac{\partial }{\partial b}\mathrm{B}_{z}\left( a,b\right)
=\int_{0}^{z}t^{a-1}\left( 1-t\right) ^{b-1}\ln \left( 1-t\right) dt.
\end{equation*}%
Perform the substitution $\tau =1-t$, and apply the formulas (\ref%
{Binomial_theorem})\ and (\ref{Int_xp_lnx}), to obtain%
\begin{eqnarray}
&&\frac{\partial }{\partial b}\mathrm{B}_{z}\left( a,b\right)
\label{DbB_sums} \\
&=&\left( 1-z\right) ^{b}\left\{ \sum_{k=0}^{\infty }\frac{\left( 1-a\right)
_{k}\left( 1-z\right) ^{k}}{k!\left( b+k\right) ^{2}}-\ln \left( 1-z\right)
\sum_{k=0}^{\infty }\frac{\left( 1-a\right) _{k}\left( 1-z\right) ^{k}}{%
k!\left( b+k\right) }\right\}   \notag \\
&&-\sum_{k=0}^{\infty }\frac{\left( 1-a\right) _{k}}{k!\left( b+k\right) ^{2}%
}.  \notag
\end{eqnarray}%
Finally, write the sums given in (\ref{DbB_sums})\ as hypergeometric
functions and apply the results given in (\ref{2F1_reduction_beta_2})\ and (%
\ref{3F2(1)})\ to arrive at (\ref{DbB_resultado}), as we wanted to prove.
\end{proof}

\subsection{Calculation of sums involving the digamma function}

\begin{theorem}
For $b\neq c+1$ and $z\in
%TCIMACRO{\U{2102} }%
%BeginExpansion
\mathbb{C}
%EndExpansion
$, $\left\vert z\right\vert <1$ the following sum holds true:%
\begin{eqnarray}
&&\sum_{k=0}^{\infty }\frac{\left( b\right) _{k}}{\left( c\right) _{k}}\psi
\left( b+k\right) \,z^{k}  \label{Reduction_b_c} \\
&=&\left( c-1\right) z^{1-c}\left\{ \frac{1}{\left( b-c+1\right) ^{2}}%
\,_{3}F_{2}\left( \left.
\begin{array}{c}
2-c,b-c+1,b-c+1 \\
b-c+2,b-c+2%
\end{array}%
\right\vert 1-z\right) \right.  \notag \\
&&+\left. \left( 1-z\right) ^{c-b-1}\left[ \left( \psi \left( b-c+1\right)
-\ln \left( 1-z\right) \right) \,\mathrm{B}\left( b-c+1,c-1\right) +\psi
\left( b\right) \,\mathrm{B}_{1-z}\left( b-c+1,c-1\right) \right] \right\}
\notag
\end{eqnarray}
\end{theorem}

\begin{proof}
On the one hand, applying the ratio test, we see that the sum given in (\ref%
{Reduction_b_c})\ converges for $\left\vert z\right\vert <1$ and diverges
for $\left\vert z\right\vert >1$. Indeed, taking
\begin{equation*}
a_{k}=\frac{\left( b\right) _{k}}{\left( c\right) _{k}}\psi \left(
b+k\right) z^{k},
\end{equation*}%
and taking into account (\ref{psi(1+z)}), we have%
\begin{eqnarray*}
\lim_{k\rightarrow \infty }\left\vert \frac{a_{k+1}}{a_{k}}\right\vert
&=&\lim_{k\rightarrow \infty }\left\vert \frac{\left( b+k\right) \psi \left(
b+k+1\right) }{\left( c+k\right) \psi \left( b+k\right) }z\right\vert \\
&=&\lim_{k\rightarrow \infty }\left\vert \frac{b+k}{c+k}\left( \frac{1}{%
\left( b+k\right) \psi \left( b+k\right) }+1\right) z\right\vert =\left\vert
z\right\vert .
\end{eqnarray*}%
On the other hand, let us differentiate both sides of the reduction formula
\cite[Eqn. 7.3.1(119)]{Prudnikov3} with respect to parameter $b$:%
\begin{eqnarray}
\sum_{k=0}^{\infty }\frac{\left( b\right) _{k}}{\left( c\right) _{k}}z^{k}
&=&\,_{2}F_{1}\left( \left.
\begin{array}{c}
1,b \\
c%
\end{array}%
\right\vert z\right)  \label{2F1_reduction_beta} \\
&=&z^{1-c}\left( 1-z\right) ^{c-b-1}\left( c-1\right) \,\mathrm{B}_{z}\left(
c-1,b-c+1\right) .  \notag
\end{eqnarray}%
Apply (\ref{D[(x)_n]})\ and (\ref{2F1_reduction_beta}) to the LHS\ of (\ref%
{2F1_reduction_beta}), to obtain:%
\begin{eqnarray}
&&\frac{\partial }{\partial b}\sum_{k=0}^{\infty }\frac{\left( b\right) _{k}%
}{\left( c\right) _{k}}z^{k}=\sum_{k=0}^{\infty }\frac{\left( b\right) _{k}}{%
\left( c\right) _{k}}\left[ \psi \left( b+k\right) -\psi \left( b\right) %
\right] z^{k}  \label{LHS} \\
&=&\sum_{k=0}^{\infty }\frac{\left( b\right) _{k}}{\left( c\right) _{k}}%
\left[ \psi \left( b+k\right) \right] z^{k}-\psi \left( b\right)
\,z^{1-c}\left( 1-z\right) ^{c-b-1}\left( c-1\right) \,\mathrm{B}_{z}\left(
c-1,b-c+1\right) .  \notag
\end{eqnarray}%
On the RHS\ of (\ref{2F1_reduction_beta}), we get
\begin{eqnarray}
&&\left( c-1\right) z^{1-c}\frac{\partial }{\partial b}\left[ \left(
1-z\right) ^{c-b-1}\,\mathrm{B}_{z}\left( c-1,b-c+1\right) \right]
\label{RHS} \\
&=&\left( c-1\right) z^{1-c}\left( 1-z\right) ^{c-b-1}\left[ -\ln \left(
1-z\right) \mathrm{B}_{z}\left( c-1,b-c+1\right) +\frac{\partial }{\partial b%
}\mathrm{B}_{z}\left( c-1,b-c+1\right) \right] .  \notag
\end{eqnarray}%
According to (\ref{DbB_resultado}), we have%
\begin{eqnarray}
&&\frac{\partial }{\partial b}\mathrm{B}_{z}\left( c-1,b-c+1\right)
\label{DbB_RHS} \\
&=&\frac{\left( 1-z\right) ^{b-c+1}}{\left( b-c+1\right) ^{2}}%
\,_{3}F_{2}\left( \left.
\begin{array}{c}
2-c,b-c+1,b-c+1 \\
b-c+2,b-c+2%
\end{array}%
\right\vert 1-z\right)  \notag \\
&&-\ln \left( 1-z\right) \,\mathrm{B}_{1-z}\left( b-c+1,c-1\right) +\mathrm{B%
}\left( c-1,b-c+1\right) \left[ \psi \left( 1+b-c\right) -\psi \left(
b\right) \right] .  \notag
\end{eqnarray}%
Now, insert (\ref{DbB_RHS})\ in (\ref{RHS})\ and apply the formula (\ref%
{Beta_reflection}), to arrive at%
\begin{eqnarray}
&&\left( c-1\right) z^{1-c}\frac{\partial }{\partial b}\left[ \left(
1-z\right) ^{c-b-1}\,\mathrm{B}_{z}\left( c-1,b-c+1\right) \right]
\label{RHS_2} \\
&=&\left( c-1\right) z^{1-c}\left\{ \left( 1-z\right) ^{c-b+1}\mathrm{B}%
\left( c-1,b-c+1\right) \left[ \psi \left( 1+b-c\right) -\psi \left(
b\right) -\ln \left( 1-z\right) \right] \right.  \notag \\
&&+\left. \frac{1}{\left( b-c+1\right) ^{2}}\,_{3}F_{2}\left( \left.
\begin{array}{c}
2-c,b-c+1,b-c+1 \\
b-c+2,b-c+2%
\end{array}%
\right\vert 1-z\right) \right\} .  \notag
\end{eqnarray}%
Finally, equate the results given in (\ref{LHS})\ and (\ref{RHS_2}), and
apply again (\ref{Beta_reflection})\ to complete the proof.
\end{proof}

\begin{remark}
It is worth noting that for $z=1$, the sum given in (\ref{Reduction_b_c})\
can be calculated taking $a=1$ in \cite{SumsPsiJL}:%
\begin{eqnarray}
&&\sum_{k=0}^{\infty }\frac{\left( a\right) _{k}\left( b\right) _{k}}{%
k!\left( c\right) _{k}}\psi \left( b+k\right)   \label{Gauss_Psi} \\
&=&\frac{\Gamma \left( c\right) \Gamma \left( c-a-b\right) }{\Gamma \left(
c-b\right) \Gamma \left( c-a\right) }\left[ \psi \left( c-b\right) -\psi
\left( c-a-b\right) +\psi \left( b\right) \right] ,  \notag \\
&&\mathrm{Re}\left( c-a-b\right) >0,  \notag
\end{eqnarray}%
thus%
\begin{eqnarray*}
&&\sum_{k=0}^{\infty }\frac{\left( b\right) _{k}}{\left( c\right) _{k}}\psi
\left( b+k\right) =\frac{c-1}{c-b-1}\left[ \frac{1}{c-b-1}+\psi \left(
b\right) \right] , \\
&&\mathrm{Re}\left( c-b\right) >1.
\end{eqnarray*}
\end{remark}

\begin{corollary}
For $z\in
%TCIMACRO{\U{2102} }%
%BeginExpansion
\mathbb{C}
%EndExpansion
$ and $\left\vert z\right\vert <1$ the following formula holds true:%
\begin{eqnarray}
&&\sum_{k=1}^{\infty }\psi \left( b+k\right) z^{k}  \label{Reduction_b_b} \\
&=&\left( b-1\right) z^{1-b}\,_{3}F_{2}\left( \left.
\begin{array}{c}
1,1,2-b \\
2,2%
\end{array}%
\right\vert 1-z\right) +\frac{z^{1-b}}{z-1}\left[ \gamma +\ln \left(
1-z\right) \right] +\frac{z^{1-b}}{z-1}\psi \left( b\right) .  \notag
\end{eqnarray}
\end{corollary}

\begin{proof}
Put apart the term for $k=0$ in (\ref{Reduction_b_c})\ and take $b=c$.
\end{proof}

\begin{corollary}
For $z\in
%TCIMACRO{\U{2102} }%
%BeginExpansion
\mathbb{C}
%EndExpansion
$, and $a\neq 1$, the following reduction formula holds true:%
\begin{equation}
_{3}F_{2}\left( \left.
\begin{array}{c}
1,1,a \\
2,2%
\end{array}%
\right\vert z\right) =\frac{\psi \left( 2-a\right) +\gamma +\ln z+\mathrm{B}%
_{1-z}\left( 2-a,0\right) }{\left( 1-a\right) z}.  \label{3F2_reduction}
\end{equation}
\end{corollary}

\begin{proof}
Taking into account (\ref{2F1_reduction_beta})\ for $c=b+1$, compare (\ref%
{Reduction_b_b})\ to the result found in the literature \cite{Miller}:%
\begin{eqnarray*}
\sum_{k=1}^{\infty }\psi \left( b+k\right) z^{k} &=&\frac{z}{1-z}\left[ \psi
\left( b\right) +\frac{1}{b}\,_{2}F_{1}\left( \left.
\begin{array}{c}
1,b \\
b+1%
\end{array}%
\right\vert z\right) \right] \\
&=&\frac{z}{1-z}\left[ \psi \left( b\right) +\frac{\mathrm{B}_{z}\left(
b,0\right) }{z^{b}}\right] ,
\end{eqnarray*}%
and solve for the $_{3}F_{2}$ function with $a=2-b$ to obtain the desired
result.
\end{proof}

\begin{remark}
It is worth noting that for $\left( 2-a\right) \in
%TCIMACRO{\U{211a} }%
%BeginExpansion
\mathbb{Q}
%EndExpansion
-\left\{ -1,-2,\ldots \right\} $, the incomplete beta function $\mathrm{B}%
_{1-z}\left( 2-a,0\right) $ given in (\ref{3F2_reduction})\ can be expressed
in terms of elementary functions \cite{BetaReductionJL}. For instance,
taking $a=3/2$ in (\ref{3F2_reduction}), and considering (\ref{psi(1/2)})\
and the formula for $n=0,1,\ldots $ \cite{BetaReductionJL}%
\begin{equation*}
\mathrm{B}_{z}\left( n+\frac{1}{2},0\right) =2\left( \tanh ^{-1}\sqrt{z}%
-\sum_{k=0}^{n-1}\frac{z^{k+1/2}}{2k+1}\right) ,
\end{equation*}%
we arrive at%
\begin{equation*}
_{3}F_{2}\left( \left.
\begin{array}{c}
1,1,\frac{3}{2} \\
2,2%
\end{array}%
\right\vert z\right) =\frac{4}{z}\ln \left( \frac{2\left( 1-\sqrt{1-z}%
\right) }{z}\right) ,
\end{equation*}%
which is given in the literature \cite[Eqn. 7.4.2(365)]{Prudnikov3}.
\end{remark}

\begin{remark}
As a consistency test, we can recover a known formula taking the limit $%
a\rightarrow 1$ in (\ref{3F2_reduction}). Indeed,
\begin{equation*}
_{3}F_{2}\left( \left.
\begin{array}{c}
1,1,1 \\
2,2%
\end{array}%
\right\vert z\right) =\lim_{a\rightarrow 1}\frac{\psi \left( 2-a\right)
+\gamma +\ln z+\mathrm{B}_{1-z}\left( 2-a,0\right) }{\left( 1-a\right) z}.
\end{equation*}%
Perform the substitution $b=1-a$, take into account (\ref%
{Beta_incomplete_def}),\ and the formula \cite[Eqn. 5.9.16]{NIST}:%
\begin{equation*}
\psi \left( z\right) +\gamma =\int_{0}^{1}\frac{1-t^{z-1}}{1-t}dt,
\end{equation*}%
to obtain%
\begin{eqnarray*}
_{3}F_{2}\left( \left.
\begin{array}{c}
1,1,1 \\
2,2%
\end{array}%
\right\vert z\right) &=&\lim_{b\rightarrow 0}\frac{\psi \left( 1+b\right)
+\gamma +\ln z+\mathrm{B}_{1-z}\left( 1+b,0\right) }{b\,z} \\
&=&\lim_{b\rightarrow 0}\frac{1}{b\,z}\left[ \ln z+\int_{0}^{1}\frac{1-t^{b}%
}{1-t}dt+\int_{0}^{1-z}\frac{t^{b}}{1-t}dt\right] .
\end{eqnarray*}%
Now, perform the susbstitution $\tau =1-t$, and apply the Taylor series:%
\begin{equation*}
\left( 1-\tau \right) ^{b}=\sum_{n=0}^{\infty }\frac{\ln ^{n}\left( 1-\tau
\right) }{n!}b^{n},
\end{equation*}%
to arrive at%
\begin{eqnarray*}
&&_{3}F_{2}\left( \left.
\begin{array}{c}
1,1,1 \\
2,2%
\end{array}%
\right\vert z\right) \\
&=&\lim_{b\rightarrow 0}\frac{1}{b\,z}\left[ \ln z-\int_{0}^{1}\frac{%
1-\left( 1-\tau \right) ^{b}}{\tau }d\tau +\int_{1}^{z}\frac{\left( 1-\tau
\right) ^{b}}{\tau }dt\right] \\
&=&\lim_{b\rightarrow 0}\frac{1}{b\,z}\left[ \ln z-\int_{0}^{1}\frac{\ln
\left( 1-\tau \right) }{\tau }b\,d\tau -\int_{1}^{z}\left( \frac{1}{\tau }+%
\frac{\ln \left( 1-\tau \right) }{\tau }b\right) dt\right] \\
&=&-\frac{1}{z}\int_{0}^{z}\frac{\ln \left( 1-\tau \right) }{\tau }d\tau .
\end{eqnarray*}%
From the following formula of the dilogarithm function \cite[Eqn. 25.12.2]%
{NIST}
\begin{equation*}
\mathrm{Li}_{2}\left( z\right) =-\int_{0}^{z}\frac{\ln \left( 1-\tau \right)
}{\tau }d\tau ,
\end{equation*}%
we recover the following result found in the literature \cite[Eqn. 7.4.2(355)%
]{Prudnikov3}:%
\begin{equation*}
_{3}F_{2}\left( \left.
\begin{array}{c}
1,1,1 \\
2,2%
\end{array}%
\right\vert z\right) =\frac{\mathrm{Li}_{2}\left( z\right) }{z}.
\end{equation*}
\end{remark}

\begin{theorem}
For $a\neq 1$ and $z\in
%TCIMACRO{\U{2102} }%
%BeginExpansion
\mathbb{C}
%EndExpansion
$, $\left\vert z\right\vert <1$ the following sum holds true:%
\begin{eqnarray}
&&\sum_{k=0}^{\infty }\frac{\left( a\right) _{k}\left( b\right) _{k}}{%
k!\left( b+1\right) _{k}}\psi \left( a+k\right) z^{k}
\label{Reduction_a_b_b+1} \\
&=&b\,z^{-b}\left\{ \left[ \ln \left( 1-z\right) -\psi \left( a\right) %
\right] \,\mathrm{B}_{1-z}\left( 1-a,b\right) +\left[ \psi \left(
1+b-a\right) -\pi \cot \left( \pi a\right) \right] \,\mathrm{B}\left(
b,1-a\right) \right.  \notag \\
&&-\left. \frac{\left( 1-z\right) ^{1-a}}{\left( 1-a\right) ^{2}}%
\,_{3}F_{2}\left( \left.
\begin{array}{c}
1-b,1-a,1-a \\
2-a,2-a%
\end{array}%
\right\vert 1-z\right) \right\} .  \notag
\end{eqnarray}
\end{theorem}

\begin{proof}
On the one hand, applying the ratio test, we see that the sum given in (\ref%
{Reduction_a_b_b+1})\ converges for $\left\vert z\right\vert <1$ and
diverges for $\left\vert z\right\vert >1$. Indeed, taking
\begin{equation*}
c_{k}=\frac{\left( a\right) _{k}\left( b\right) _{k}}{k!\left( b+1\right)
_{k}}\psi \left( a+k\right) z^{k},
\end{equation*}%
and taking into account (\ref{psi(1+z)}), we have%
\begin{eqnarray*}
\lim_{k\rightarrow \infty }\left\vert \frac{c_{k+1}}{c_{k}}\right\vert
&=&\lim_{k\rightarrow \infty }\left\vert \frac{\left( a+k\right) \left(
b+k\right) \psi \left( a+k+1\right) }{\left( k+1\right) \left( b+1+k\right)
\psi \left( a+k\right) }z\right\vert \\
&=&\lim_{k\rightarrow \infty }\left\vert \frac{\left( a+k\right) \left(
b+k\right) }{\left( k+1\right) \left( b+1+k\right) }\left( \frac{1}{\left(
a+k\right) \psi \left( a+k\right) }+1\right) z\right\vert =\left\vert
z\right\vert .
\end{eqnarray*}%
On the other hand, taking into account (\ref{D[(x)_n]}), differentiate the
reduction formula (\ref{2F1_reduction_beta_2}), i.e.%
\begin{equation*}
_{2}F_{1}\left( \left.
\begin{array}{c}
a,b \\
b+1%
\end{array}%
\right\vert z\right) =\sum_{k=0}^{\infty }\frac{\left( a\right) _{k}\left(
b\right) _{k}}{k!\left( b+1\right) _{k}}z^{k}=b\,z^{-b}\,\mathrm{B}%
_{z}\left( b,1-a\right) ,
\end{equation*}%
with respect to the parameter $a$ to obtain,
\begin{equation}
\sum_{k=0}^{\infty }\frac{\left( a\right) _{k}\left( b\right) _{k}}{k!\left(
b+1\right) _{k}}\left[ \psi \left( a+k\right) -\psi \left( a\right) \right]
z^{k}=b\,z^{-b}\frac{\partial }{\partial a}\mathrm{B}_{z}\left( b,1-a\right)
.  \label{Sum_reduction_a_b_b+1}
\end{equation}%
Apply (\ref{2F1_reduction_beta_2})\ on the LHS\ of (\ref%
{Sum_reduction_a_b_b+1}), and (\ref{DbB_resultado})\ on the RHS\ of (\ref%
{Sum_reduction_a_b_b+1})\ to arrive at,%
\begin{eqnarray}
&&\sum_{k=0}^{\infty }\frac{\left( a\right) _{k}\left( b\right) _{k}}{%
k!\left( b+1\right) _{k}}\psi \left( a+k\right) z^{k}
\label{Reduction_a_b_b+1_quasi} \\
&=&b\,z^{-b}\left\{ \ln \left( 1-z\right) \,\mathrm{B}_{1-z}\left(
1-a,b\right) +\mathrm{B}\left( b,1-a\right) \left[ \psi \left( 1+b-a\right)
-\psi \left( 1-a\right) \right] \right.  \notag \\
&&+\left. \psi \left( a\right) \,\mathrm{B}_{z}\left( b,1-a\right) -\frac{%
\left( 1-z\right) ^{1-a}}{\left( 1-a\right) ^{2}}\,_{3}F_{2}\left( \left.
\begin{array}{c}
1-b,1-a,1-a \\
2-a,2-a%
\end{array}%
\right\vert 1-z\right) \right\} .  \notag
\end{eqnarray}%
Finally, apply (\ref{Beta_reflection})\ and (\ref{psi(1-z)-psi(z)})\ in
order to reduce (\ref{Reduction_a_b_b+1_quasi})\ to (\ref{Reduction_a_b_b+1}%
), as we wanted to prove.
\end{proof}

\begin{remark}
It is worth noting that for $z=1$, the sum given in (\ref{Reduction_a_b_b+1}%
)\ can be calculated taking $c=b+1$ in (\ref{Gauss_Psi}), and applying (\ref%
{Gamma_factorial}), (\ref{Beta_property})\ and (\ref{psi(1-z)-psi(z)}), to
obtain:%
\begin{eqnarray*}
&&\sum_{k=0}^{\infty }\frac{\left( a\right) _{k}\left( b\right) _{k}}{%
k!\left( b+1\right) _{k}}\psi \left( b+k\right) =b\,\mathrm{B}\left(
b,1-a\right) \left[ \psi \left( 1+b-a\right) -\pi \cot \left( \pi a\right) %
\right] \,, \\
&&\mathrm{Re\,}a<1.
\end{eqnarray*}
\end{remark}

\begin{corollary}
For $a\neq 1$, and $z\in
%TCIMACRO{\U{2102} }%
%BeginExpansion
\mathbb{C}
%EndExpansion
$, $\left\vert z\right\vert <1$ the following formula holds true:%
\begin{eqnarray}
&&\sum_{k=0}^{\infty }\frac{\left( a\right) _{k}}{\left( k+1\right) !}\psi
\left( a+k\right) z^{k}=\frac{1}{\left( 1-a\right) z}
\label{Reduction_a_k+1} \\
&&\left\{ \left( 1-z\right) ^{1-a}\left[ \ln \left( 1-z\right) -\psi \left(
a\right) +\frac{1}{a-1}\right] +\psi \left( 2-a\right) -\pi \cot \left( \pi
a\right) \right\} .  \notag
\end{eqnarray}
\end{corollary}

\begin{proof}
Take $b=1$ in (\ref{Reduction_a_b_b+1}) and consider that for $a\neq 1$ we
have
\begin{eqnarray*}
\,\mathrm{B}_{1-z}\left( 1-a,1\right) &=&\frac{\left( 1-z\right) ^{1-a}}{1-a}%
, \\
\mathrm{B}\left( b,1-a\right) &=&\frac{1}{1-a}.
\end{eqnarray*}
\end{proof}

\section{Sums connected to Bessel functions\label{Section: Bessel}}

If we differenciate the following sum formulas \cite[Eqn. 7.13.1(1)]%
{Prudnikov3}:%
\begin{equation*}
_{0}F_{1}\left( \left.
\begin{array}{c}
- \\
b%
\end{array}%
\right\vert -z\right) =\sum_{k=0}^{\infty }\frac{\left( -z\right) ^{k}}{%
k!\left( b\right) _{k}}=\Gamma \left( b\right) \,z^{\left( 1-b\right)
/2}\,J_{b-1}\left( 2\sqrt{z}\right) ,
\end{equation*}%
and%
\begin{equation*}
_{0}F_{1}\left( \left.
\begin{array}{c}
- \\
b%
\end{array}%
\right\vert z\right) =\sum_{k=0}^{\infty }\frac{z^{k}}{k!\left( b\right) _{k}%
}=\Gamma \left( b\right) \,z^{\left( 1-b\right) /2}\,I_{b-1}\left( 2\sqrt{z}%
\right) ,
\end{equation*}%
with respect to parameter $b$, taking into account (\ref{D[1/(x)_n]}), we
obtain:%
\begin{eqnarray}
&&\sum_{k=0}^{\infty }\frac{\left( -z\right) ^{k}\psi \left( k+b\right) }{%
k!\left( b\right) _{k}}  \label{Sum_J} \\
&=&z^{\left( 1-b\right) /2}\Gamma \left( b\right) \left[ J_{b-1}\left( 2%
\sqrt{z}\right) \ln \sqrt{z}-\frac{\partial J_{b-1}\left( 2\sqrt{z}\right) }{%
\partial b}\right] ,  \notag
\end{eqnarray}%
and%
\begin{eqnarray}
&&\sum_{k=0}^{\infty }\frac{z^{k}\psi \left( k+b\right) }{k!\left( b\right)
_{k}}  \label{Sum_I} \\
&=&z^{\left( 1-b\right) /2}\Gamma \left( b\right) \left[ I_{b-1}\left( 2%
\sqrt{z}\right) \ln \sqrt{z}-\frac{\partial I_{b-1}\left( 2\sqrt{z}\right) }{%
\partial b}\right] ,  \notag
\end{eqnarray}%
which are found in an equivalent form in \cite[Eqns. 55.7.11-12]{Hansen}.
For $b=n\in
%TCIMACRO{\U{2115} }%
%BeginExpansion
\mathbb{N}
%EndExpansion
$, we found in \cite{Miller} closed-form expressions for (\ref{Sum_J}) and (%
\ref{Sum_I}). We can obtain closed-form expressions for other values of $b$
using the following formulas \cite{DerivativesBesselJL} for $\nu \geq 0$, $%
\mathrm{Re\,}z>0$:%
\begin{eqnarray}
\frac{\partial J_{\nu }\left( z\right) }{\partial \nu } &=&\frac{\pi }{2}%
\left[ \frac{Y_{\nu }\left( z\right) \left( z/2\right) ^{2\nu }}{\,\Gamma
^{2}\left( \nu +1\right) }\,_{2}F_{3}\left( \left.
\begin{array}{c}
\nu ,\nu +\frac{1}{2} \\
2\nu +1,\nu +1,\nu +1%
\end{array}%
\right\vert -z^{2}\right) \right.  \label{DJnu_Meijer} \\
&&-\left. \frac{\nu J_{\nu }\left( z\right) }{\sqrt{\pi }}%
G_{2,4}^{3,0}\left( z^{2}\left\vert
\begin{array}{c}
\frac{1}{2},1 \\
0,0,\nu ,-\nu%
\end{array}%
\right. \right) \right] ,  \notag
\end{eqnarray}%
and
\begin{eqnarray}
\frac{\partial I_{\nu }\left( z\right) }{\partial \nu } &=&\frac{-\nu I_{\nu
}\left( z\right) }{2\sqrt{\pi }}G_{2,4}^{3,1}\left( z^{2}\left\vert
\begin{array}{c}
\frac{1}{2},1 \\
0,0,\nu ,-\nu%
\end{array}%
\right. \right)  \label{DInu_Meijer} \\
&&-\frac{K_{\nu }\left( z\right) \left( z/2\right) ^{2\nu }}{\Gamma
^{2}\left( \nu +1\right) }\,_{2}F_{3}\left( \left.
\begin{array}{c}
\nu ,\nu +\frac{1}{2} \\
2\nu +1,\nu +1,\nu +1%
\end{array}%
\right\vert z^{2}\right) .\quad  \notag
\end{eqnarray}

\begin{theorem}
For $b\geq 1$ and $\mathrm{Re\,}z>0$, the following sum holds true:
\begin{eqnarray}
&&\sum_{k=0}^{\infty }\frac{\left( -z\right) ^{k}\psi \left( k+b\right) }{%
k!\left( b\right) _{k}}  \label{Sum_J_resultado} \\
&=&\frac{z^{-\left( 1+b\right) /2}}{8\,\Gamma \left( b\right) }  \notag \\
&&\left\{ \Gamma ^{2}\left( b\right) J_{b-1}\left( 2\sqrt{z}\right) \left[
\sqrt{\pi }\left( b-1\right) G_{2,4}^{3,0}\left( 4z\left\vert
\begin{array}{c}
\frac{3}{2},2 \\
1,1,b,2-b%
\end{array}%
\right. \right) +4z\ln z\right] \right.  \notag \\
&&-\left. 4\pi \,z^{b}Y_{b-1}\left( 2\sqrt{z}\right) \,_{2}F_{3}\left(
\left.
\begin{array}{c}
b-1,b-\frac{1}{2} \\
b,b,2b-1%
\end{array}%
\right\vert -4z\right) \right\} .  \notag
\end{eqnarray}
\end{theorem}

\begin{proof}
Calculate (\ref{Sum_J})\ taking into account (\ref{DJnu_Meijer})\ to arrive
at the desired result.
\end{proof}

\begin{theorem}
For $b\geq 1$ and $\mathrm{Re\,}z>0$, the following sum holds true:%
\begin{eqnarray}
&&\sum_{k=0}^{\infty }\frac{z^{k}\psi \left( k+b\right) }{k!\left( b\right)
_{k}}  \label{Sum_I_resultado} \\
&=&\frac{z^{-\left( 1+b\right) /2}}{8\,\sqrt{\pi }\Gamma \left( b\right) }
\notag \\
&&\left\{ \Gamma ^{2}\left( b\right) I_{b-1}\left( 2\sqrt{z}\right) \left[
\left( b-1\right) G_{2,4}^{3,1}\left( 4z\left\vert
\begin{array}{c}
\frac{3}{2},2 \\
1,1,b,2-b%
\end{array}%
\right. \right) +4\sqrt{\pi }z\ln z\right] \right.  \notag \\
&&+\left. 8\sqrt{\pi }\,z^{b}K_{b-1}\left( 2\sqrt{z}\right)
\,_{2}F_{3}\left( \left.
\begin{array}{c}
b-1,b-\frac{1}{2} \\
b,b,2b-1%
\end{array}%
\right\vert 4z\right) \right\} .  \notag
\end{eqnarray}
\end{theorem}

\begin{proof}
Calculate (\ref{Sum_I})\ taking into account (\ref{DInu_Meijer})\ to arrive
at the desired result.
\end{proof}

\begin{theorem}
For $b\geq 1$ and $\mathrm{Re\,}z>0$, the following sum holds true:%
\begin{eqnarray}
&&\sum_{k=0}^{\infty }\frac{z^{k}\psi \left( 2k+b\right) }{k!\left( \frac{1}{%
2}\right) _{k}\left( \frac{b}{2}\right) _{k}\left( \frac{b+1}{2}\right) _{k}}
\label{Sum_J+I} \\
&=&\frac{\Gamma \left( b\right) }{2^{b}z^{\left( b-1\right) /4}}  \notag \\
&&\left\{ \ln \left( 2\,z^{1/4}\right) \left[ J_{b-1}\left(
4\,z^{1/4}\right) +I_{b-1}\left( 4\,z^{1/4}\right) \right] -\frac{\partial
J_{b-1}\left( 4\,z^{1/4}\right) }{\partial b}-\frac{\partial I_{b-1}\left(
4\,z^{1/4}\right) }{\partial b}\right\} ,  \notag
\end{eqnarray}%
where the order derivatives of the Bessel functions are calculated using (%
\ref{DJnu_Meijer})\ and (\ref{DInu_Meijer}).
\end{theorem}

\begin{proof}
Sum up (\ref{Sum_J})\ and (\ref{Sum_I}), using the duplication formula of
the gamma function (\ref{Gamma_duplication}), to arrrive at the desired
result.
\end{proof}

\begin{theorem}
For $b\geq 1$ and $\mathrm{Re\,}z>0$, the following sum holds true:%
\begin{eqnarray}
&&\sum_{k=0}^{\infty }\frac{z^{k}\psi \left( 2k+b\right) }{k!\left( \frac{3}{%
2}\right) _{k}\left( \frac{b}{2}\right) _{k}\left( \frac{b+1}{2}\right) _{k}}
\label{Sum_J-I} \\
&=&\frac{\Gamma \left( b\right) }{2^{b+1}z^{b/4}}  \notag \\
&&\left\{ \ln \left( 2\,z^{1/4}\right) \left[ I_{b-2}\left(
4\,z^{1/4}\right) -J_{b-2}\left( 4\,z^{1/4}\right) \right] -\frac{\partial
I_{b-2}\left( 4\,z^{1/4}\right) }{\partial b}+\frac{\partial J_{b-2}\left(
4\,z^{1/4}\right) }{\partial b}\right\} ,  \notag
\end{eqnarray}%
where the order derivatives of the Bessel functions are calculated using (%
\ref{DJnu_Meijer})\ and (\ref{DInu_Meijer}).
\end{theorem}

\begin{proof}
Substract (\ref{Sum_I})\ from (\ref{Sum_J}), and apply again (\ref%
{Gamma_duplication}), to arrrive at the desired result.
\end{proof}

\section{Application to the parameter derivative of some special functions
\label{Section: Application}}

\subsection{Application to the derivative of the Wright function with
respect to the parameters}

The Wright function is defined as \cite[Eqn. 10.46.1]{NIST}:%
\begin{equation*}
\mathrm{W}_{\alpha ,\beta }\left( z\right) =\sum_{k=0}^{\infty }\frac{%
\,\,z^{k}}{k!\,\Gamma \left( \alpha k+\beta \right) },\quad \alpha >-1,
\end{equation*}%
thus,
\begin{eqnarray}
\frac{\partial \mathrm{W}_{\alpha ,\beta }\left( z\right) }{\partial \alpha }
&=&-\sum_{k=1}^{\infty }\frac{\,k\,z^{k}\psi \left( \alpha k+\beta \right) }{%
k!\,\Gamma \left( \alpha k+\beta \right) },  \label{DaW} \\
\frac{\partial \mathrm{W}_{\alpha ,\beta }\left( z\right) }{\partial \beta }
&=&-\sum_{k=0}^{\infty }\frac{\,z^{k}\psi \left( \alpha k+\beta \right) }{%
k!\,\Gamma \left( \alpha k+\beta \right) },  \label{DbW}
\end{eqnarray}%
and the following equation is satisfied:%
\begin{equation}
\frac{\partial \mathrm{W}_{\alpha ,\beta }\left( z\right) }{\partial \alpha }%
=z\frac{\partial }{\partial z}\left( \frac{\partial \mathrm{W}_{\alpha
,\beta }\left( z\right) }{\partial \beta }\right) .  \label{DaW=z*DzDbW}
\end{equation}

In \cite{Apelblat3}, we found some reduction formulas for the first
derivative of the Wright function with respect to the parameters for
particular values of $\alpha $ and $\beta $. Next, we extend these reduction
formulas. For this purpose, apply (\ref{Sum_I_resultado}) to arrive at the
following result:

\begin{theorem}
For $\beta \geq 1\,$and $\mathrm{Re\,}z>0$, we have \
\begin{eqnarray}
&&\left. \frac{\partial \mathrm{W}_{\alpha ,\beta }\left( z\right) }{%
\partial \beta }\right\vert _{\alpha =1}  \label{DbWa1b} \\
&=&z^{-\left( 1+\beta \right) /2}\left\{ I_{\beta -1}\left( 2\sqrt{z}\right) %
\left[ \frac{1-\beta }{8\sqrt{\pi }}G_{2,4}^{3,1}\left( 4z\left\vert
\begin{array}{c}
3/2,2 \\
1,1,\beta ,2-\beta%
\end{array}%
\right. \right) -\frac{z}{2}\ln z\right] \right.  \notag \\
&&\qquad -\left. \frac{z^{\beta }}{\Gamma ^{2}\left( \beta \right) }K_{\beta
-1}\left( 2\sqrt{z}\right) \,_{2}F_{3}\left( \left.
\begin{array}{c}
\beta -1,\beta -\frac{1}{2} \\
\beta ,\beta ,2\beta -1%
\end{array}%
\right\vert 4z\right) \right\} .  \notag
\end{eqnarray}
\end{theorem}

\begin{remark}
It is worth noting that for $\beta =1$, Eqn. (\ref{DbWa1b})\ is reduced to
\begin{equation*}
\left. \frac{\partial \mathrm{W}_{\alpha ,\beta }\left( z\right) }{\partial
\beta }\right\vert _{\alpha =\beta =1}=-\frac{1}{2}\ln z\,I_{0}\left( 2\sqrt{%
z}\right) -K_{0}\left( 2\sqrt{z}\right) ,
\end{equation*}%
which is found in \cite{Apelblat3}.
\end{remark}

Further, from (\ref{DaW=z*DzDbW})\ and (\ref{DbWa1b}) and with the aid of
MATHEMATICA program, we arrive at the following result:

\begin{theorem}
For $\beta \geq 1\,$and $\mathrm{Re\,}z>0$, we have%
\begin{eqnarray}
&&\left. \frac{\partial \mathrm{W}_{\alpha ,\beta }\left( z\right) }{%
\partial \alpha }\right\vert _{\alpha =1}  \label{DaWa1b} \\
&=&\frac{z^{-\left( \beta +1\right) /2}}{2}\left\{ \frac{\beta -1}{8\sqrt{%
\pi }}\left\{ \left( \beta -1\right) I_{\beta -1}\left( 2\sqrt{z}\right) -%
\sqrt{z}\left[ I_{\beta -2}\left( 2\sqrt{z}\right) +I_{\beta }\left( 2\sqrt{z%
}\right) \right] \right\} G_{2,4}^{3,1}\left( 4z\left\vert
\begin{array}{c}
3/2,2 \\
1,1,\beta ,2-\beta%
\end{array}%
\right. \right) \right.  \notag \\
&&+\frac{z^{\beta }}{\Gamma ^{2}\left( \beta \right) }\left\{ \left( \beta
-1\right) K_{\beta -1}\left( 2\sqrt{z}\right) +\sqrt{z}\left[ K_{\beta
-2}\left( 2\sqrt{z}\right) +K_{\beta }\left( 2\sqrt{z}\right) \right]
\right\} \,_{2}F_{3}\left( \left.
\begin{array}{c}
\beta -1,\beta -\frac{1}{2} \\
\beta ,\beta ,2\beta -1%
\end{array}%
\right\vert 4z\right)  \notag \\
&&+\frac{I_{\beta -1}\left( 2\sqrt{z}\right) }{4\sqrt{\pi }}\left[ 2\sqrt{%
\pi }z\left[ \left( \beta -1\right) \ln z-2\right] +\left( \beta -1\right)
G_{1,3}^{2,1}\left( 4z\left\vert
\begin{array}{c}
3/2 \\
1,\beta ,2-\beta%
\end{array}%
\right. \right) \right]  \notag \\
&&-\left. \frac{z^{3/2}\ln z}{2}\left[ I_{\beta -2}\left( 2\sqrt{z}\right)
+I_{\beta }\left( 2\sqrt{z}\right) \right] +\frac{2\left( 1-\beta \right)
z^{\beta }}{\Gamma ^{2}\left( \beta \right) }K_{\beta -1}\left( 2\sqrt{z}%
\right) \,_{1}F_{2}\left( \left.
\begin{array}{c}
\beta -\frac{1}{2} \\
\beta ,2\beta -1%
\end{array}%
\right\vert 4z\right) \right\} .  \notag
\end{eqnarray}
\end{theorem}

\begin{remark}
It is worth noting that for $\beta =1$, Eqn. (\ref{DaWa1b})\ is reduced to
\begin{equation*}
\left. \frac{\partial \mathrm{W}_{\alpha ,\beta }\left( z\right) }{\partial
\beta }\right\vert _{\alpha =\beta =1}=\frac{\sqrt{z}\left[ K_{1}\left( 2%
\sqrt{z}\right) -\ln z\,I_{1}\left( 2\sqrt{z}\right) \right] -I_{0}\left( 2%
\sqrt{z}\right) }{2},
\end{equation*}%
which is also found in \cite{Apelblat3}.
\end{remark}

\subsection{Application to the derivative of the Mittag-Leffler function
with respect to the parameters}

The two-parameter Mittag-Leffler function is defined as \cite[Eqn. 10.46.3]%
{NIST}:%
\begin{equation}
\mathrm{E}_{\alpha ,\beta }\left( z\right) =\sum_{k=0}^{\infty }\frac{\,z^{k}%
}{\Gamma \left( \alpha k+\beta \right) },\quad \alpha >0,  \label{ML_def}
\end{equation}%
thus,
\begin{eqnarray}
\frac{\partial \mathrm{E}_{\alpha ,\beta }\left( z\right) }{\partial \alpha }
&=&-\sum_{k=0}^{\infty }\frac{k\,z^{k}\psi \left( \alpha k+\beta \right) }{%
\Gamma \left( \alpha k+\beta \right) },  \label{DaML_series} \\
\frac{\partial \mathrm{E}_{\alpha ,\beta }\left( z\right) }{\partial \beta }
&=&-\sum_{k=0}^{\infty }\frac{\,z^{k}\psi \left( \alpha k+\beta \right) }{%
\Gamma \left( \alpha k+\beta \right) },  \label{DbML_series}
\end{eqnarray}%
and%
\begin{equation}
\frac{\partial \mathrm{E}_{\alpha ,\beta }\left( z\right) }{\partial \alpha }%
=z\frac{\partial }{\partial z}\left( \frac{\partial \mathrm{E}_{\alpha
,\beta }\left( z\right) }{\partial \beta }\right) .  \label{DaML=zDz[DbML]}
\end{equation}

For this purpose, consider the following functions:

\begin{definition}
According to \cite[Eqn. 6.2.1(63)]{Brychov}, define
\begin{eqnarray}
&&Q\left( a,t\right) =\sum_{k=0}^{\infty }\frac{t^{k}}{\left( a\right) _{k}}%
\psi \left( k+a\right)  \label{Q(t,a)_def} \\
&=&\psi \left( a\right) +e^{t}\left[ t^{1-a}\psi \left( a\right) \gamma
\left( a,t\right) +\frac{t}{a^{2}}\,_{2}F_{2}\left( \left.
\begin{array}{c}
a,a \\
a+1,a+1%
\end{array}%
\right\vert -t\right) \right] ,  \notag
\end{eqnarray}%
thus%
\begin{eqnarray}
&&P\left( a,t\right) =\frac{\partial Q\left( a,t\right) }{\partial t}
\label{P(t,a)_def} \\
&=&\psi \left( a\right) +e^{t}  \notag \\
&&\left\{ \frac{t-a+1}{a^{2}}\,_{2}F_{2}\left( \left.
\begin{array}{c}
a,a \\
a+1,a+1%
\end{array}%
\right\vert -t\right) +t^{-a}\gamma \left( a,t\right) \left[ 1+\left(
t-a+1\right) \psi \left( a\right) \right] \right\} .  \notag
\end{eqnarray}
\end{definition}

In \cite{Apelblat3}, we found some reduction formulas of the first
derivative of the Mittag-Leffler function with respect to the parameters for
particular values of $\alpha $ and $\beta $. In particular, we found for $%
q=1,2,\ldots $ that%
\begin{eqnarray}
&&\left. \frac{\partial \mathrm{E}_{\alpha ,\beta }\left( z\right) }{%
\partial \alpha }\right\vert _{\alpha =1/q}  \label{DaML_1/q} \\
&=&-\sum_{h=0}^{q-1}\frac{h\left[ \psi \left( \frac{h}{q}+\beta \right) +%
\tilde{Q}\left( \frac{h}{q}+\beta ,z^{q}\right) +q\,z^{q}\,P\left( \frac{h}{q%
}+\beta ,z^{q}\right) \right] }{\Gamma \left( \frac{h}{q}+\beta \right) }%
z^{h},  \notag
\end{eqnarray}%
and%
\begin{equation}
\left. \frac{\partial \mathrm{E}_{\alpha ,\beta }\left( z\right) }{\partial
\beta }\right\vert _{\alpha =1/q}=-\sum_{h=0}^{q-1}\frac{\psi \left( \frac{h%
}{q}+\beta \right) +\tilde{Q}\left( \frac{h}{q}+\beta ,z^{q}\right) }{\Gamma
\left( \frac{h}{q}+\beta \right) }z^{h},  \label{DbML_1/q}
\end{equation}%
where%
\begin{equation*}
\tilde{Q}\left( a,t\right) =Q\left( a,t\right) -\psi \left( a\right) .
\end{equation*}

Next, we extend these reduction formulas to other values of the parameters.
For this purpose, consider the following lemma.

\begin{lemma}
\label{Lemma_sum}For $n=1,2,\ldots $, the following sum identity holds true:%
\begin{equation}
\sum_{k=0}^{\infty }a_{nk}=\sum_{k=0}^{\infty }\theta _{n,k}\,a_{k},
\label{sum_a_nk}
\end{equation}%
where
\begin{equation}
\theta _{n,k}=\frac{1}{n}\sum_{m=1}^{n}\exp \left( \frac{2\pi \,i\,m\,k}{n}%
\right) .  \label{theta_n,k_def}
\end{equation}
\end{lemma}

\begin{theorem}
For $n=1,2,\ldots $, the following reduction formula holds true:\
\begin{equation}
\left. \frac{\partial \mathrm{E}_{\alpha ,\beta }\left( z\right) }{\partial
\beta }\right\vert _{\alpha =n}=-\frac{1}{n\,\Gamma \left( \beta \right) }%
\sum_{m=1}^{n}Q\left( \beta ,z^{1/n}e^{i2\pi m/n}\right) .
\label{DbML_n_resultado}
\end{equation}
\end{theorem}

\begin{proof}
According to (\ref{DbML_series}) and Lemma \ref{Lemma_sum}, we have%
\begin{eqnarray*}
\left. \frac{\partial \mathrm{E}_{\alpha ,\beta }\left( z\right) }{\partial
\beta }\right\vert _{\alpha =n} &=&-\sum_{k=0}^{\infty }\frac{\,z^{k}\psi
\left( nk+\beta \right) }{\Gamma \left( nk+\beta \right) } \\
&=&-\sum_{k=0}^{\infty }\theta _{n,k}\,\frac{\,z^{k/n}\psi \left( k+\beta
\right) }{\Gamma \left( k+\beta \right) } \\
&=&-\frac{1}{n\,\Gamma \left( \beta \right) }\sum_{m=1}^{n}\exp \left( \frac{%
2\pi \,i\,m\,k}{n}\right) \sum_{k=0}^{\infty }\frac{\,z^{k/n}\psi \left(
k+\beta \right) }{\left( \beta \right) _{k}}.
\end{eqnarray*}%
Finally, take into account (\ref{Q(t,a)_def})\ to arrive at the desired
result.
\end{proof}

\begin{theorem}
For $n=1,2,\ldots $, the following reduction formula holds true:%
\begin{equation}
\left. \frac{\partial \mathrm{E}_{\alpha ,\beta }\left( z\right) }{\partial
\alpha }\right\vert _{\alpha =n}=-\frac{z^{1/n}}{n^{2}\,\Gamma \left( \beta
\right) }\sum_{m=1}^{n}e^{i2\pi m/n}P\left( \beta ,z^{1/n}e^{i2\pi
m/n}\right) .  \label{DaML_n_resultado}
\end{equation}
\end{theorem}

\begin{proof}
Apply (\ref{DaML=zDz[DbML]})\ to (\ref{DbML_n_resultado})\ and take into
account the definition given in (\ref{P(t,a)_def}).
\end{proof}

\begin{remark}
It is worth noting that for $\alpha =1$, (\ref{DaML_n_resultado})\ is
equivalent to (\ref{DaML_1/q}), and (\ref{DbML_n_resultado})\ is equivalent
to (\ref{DbML_1/q}).
\end{remark}

For particular values of $\alpha $ and $\beta $, the first derivative of the
Mittag-Leffler function with respect to the parameters are shown in Tables %
\ref{Table_DaMLn}\ and \ref{Table_DbMLn}, using the results given in (\ref%
{DbML_n_resultado})\ and (\ref{DaML_n_resultado}) with the aid of
MATHEMATICA program.

\begin{center}
%TCIMACRO{\TeXButton{B}{\begin{table}[tbp] \centering}}%
%BeginExpansion
\begin{table}[htbp] \centering%
%EndExpansion
\begin{tabular}{|c|c|c|}
\hline
$\alpha $ & $\beta $ & $\frac{\partial \mathrm{E}_{\alpha ,\beta }\left(
z\right) }{\partial \alpha }$ \\ \hline\hline
$1$ & $1$ & $1-e^{z}\left\{ z\left[ \ln z+\Gamma \left( 0,z\right) \right]
+1\right\} $ \\ \hline
$1$ & $2$ & $\frac{1}{z}\left\{ 1+\gamma -e^{z}\left[ 1+\left( z-1\right)
\left( \ln z+\Gamma \left( 0,z\right) \right) \right] \right\} $ \\ \hline
$2$ & $1$ & $\frac{1}{8}e^{-\sqrt{z}}\left\{ \sqrt{z}\left[ \ln z-2\,\mathrm{%
Ei}\left( \sqrt{z}\right) -e^{2\sqrt{z}}\left( 2\,\mathrm{E}_{1}\left( \sqrt{%
z}\right) +\ln z\right) \right] -2\left( e^{\sqrt{z}}-1\right) ^{2}\right\} $
\\ \hline
$2$ & $2$ & $\frac{1}{8\sqrt{z}}e^{-\sqrt{z}}\left\{ e^{2\sqrt{z}}\left[
\left( 1-\sqrt{z}\right) \left( 2\,\mathrm{E}_{1}\left( \sqrt{z}\right) +\ln
z\right) -2\right] +\left( 1+\sqrt{z}\right) \left[ 2\,\mathrm{Ei}\left(
\sqrt{z}\right) -\ln z\right] +2\right\} $ \\ \hline
\end{tabular}%
\caption{First derivative of the Mittag-Leffler function with respect to
$\alpha$.}\label{Table_DaMLn}%
%TCIMACRO{\TeXButton{E}{\end{table}}}%
%BeginExpansion
\end{table}%
%EndExpansion
\end{center}

\bigskip

\begin{center}
%TCIMACRO{\TeXButton{B}{\begin{table}[tbp] \centering}}%
%BeginExpansion
\begin{table}[htbp] \centering%
%EndExpansion
\begin{tabular}{|c|c|c|}
\hline
$\alpha $ & $\beta $ & $\frac{\partial \mathrm{E}_{\alpha ,\beta }\left(
z\right) }{\partial \beta }$ \\ \hline\hline
$1$ & $1$ & $-e^{z}\left[ \ln z+\Gamma \left( 0,z\right) \right] $ \\ \hline
$1$ & $2$ & $-\frac{1}{z}\left\{ e^{z}\left[ \ln z+\Gamma \left( 0,z\right) %
\right] +\gamma \right\} $ \\ \hline
$2$ & $1$ & $\frac{1}{4}e^{-\sqrt{z}}\left\{ 2\,\mathrm{Ei}\left( \sqrt{z}%
\right) -\ln z-e^{2\sqrt{z}}\left[ \ln z+2\,\Gamma \left( 0,\sqrt{z}\right) %
\right] \right\} $ \\ \hline
$2$ & $2$ & $\frac{1}{4\sqrt{z}}e^{-\sqrt{z}}\left\{ \ln z-2\,\mathrm{Ei}%
\left( \sqrt{z}\right) -e^{2\sqrt{z}}\left[ \ln z+2\,\Gamma \left( 0,\sqrt{z}%
\right) \right] \right\} $ \\ \hline
\end{tabular}%
\caption{First derivative of the Mittag-Leffler function with respect to
$\beta$.}\label{Table_DbMLn}%
%TCIMACRO{\TeXButton{E}{\end{table}}}%
%BeginExpansion
\end{table}%
%EndExpansion
\end{center}

\section{Conclusions \label{Section: Conclusions}}

We have calculated some new infinite sums involving the digamma function. On
the one hand, some of these new sums are connected to the incomplete beta
function, i.e. Eqns. (\ref{Reduction_b_c}) and (\ref{Reduction_a_b_b+1}).
For this purpose, we have derived a new $_{3}F_{2}$ hypergeometric sum at
argument unity in (\ref{3F2(1)}). Also, we have calculated in (\ref%
{DaB_resultado})\ and (\ref{DbB_resultado})\ new expressions for the
derivatives of the incomplete beta function $\mathrm{B}_{z}\left( a,b\right)
$ with respect to the parameters $a$ an $b$. As a consequence of the latter,
we have obtained in (\ref{Integral_Lerch_resultado})\ a definite integral
which does not seem to be tabulated in the most common literature. Also, in (%
\ref{3F2_reduction})\ we have derived a new reduction formula for a $%
_{3}F_{2}$ hypergeometric function.

On the other hand, we have calculated sums involving the digamma function
which are connected to the Bessel functions, i.e. Eqns (\ref{Sum_J_resultado}%
), (\ref{Sum_I_resultado}), (\ref{Sum_J+I})\ and (\ref{Sum_J-I}). For this
purpose, we have used the derivative of the Pochhammer symbol given in (\ref%
{D[1/(x)_n]}), as well as some expressions found in the literature for the
order derivatives of $J_{\nu }\left( z\right) $ and $I_{\nu }\left( z\right)
$, given in (\ref{DJnu_Meijer})\ and (\ref{DInu_Meijer})\ respectively.

Finally, we have calculated some reduction formulas for the derivatives of
some special functions with respect to the parameters as an application of
the sums involving the digamma function. In particular, we have applied the
sum presented in (\ref{Sum_I_resultado})\ to the calculation of the
reduction formulas (\ref{DbWa1b})\ and (\ref{DaWa1b})\ for the derivatives
of the Wright function with respect to the parameters. Similarly, applying
the sum given in (\ref{Q(t,a)_def}), we have calculated the reduction
formulas (\ref{DbML_n_resultado})\ and (\ref{DaML_n_resultado})\ for the
derivatives of the Mittag-Leffler function with respect to the parameters.

%\bibliographystyle{plain}
%\bibliography{My_bib_New_sums_Psi_B}

\begin{thebibliography}{10}

\bibitem{Apelblat1}
A~Apelblat.
\newblock Differentiation of the {M}ittag-{L}effler functions with respect to
  parameters in the {L}aplace transform approach.
\newblock {\em Mathematics}, 8(5):657, 2020.

\bibitem{Apelblat2}
A~Apelblat and JL~Gonz{\'a}lez-Santander.
\newblock The {I}ntegral {M}ittag-{L}effler, {W}hittaker and {W}right
  functions.
\newblock {\em Mathematics}, 9(24):3255, 2021.

\bibitem{Apelblat3}
A~Apelblat and JL~Gonz{\'a}lez-Santander.
\newblock Differentiation of integral {M}ittag-{L}effler and integral {W}right
  functions with respect to parameters.
\newblock {\em Fractional Calculus and Applied Analysis}, 2023.

\bibitem{Brychov}
YuA Brychkov.
\newblock {\em Handbook of {S}pecial {F}unctions: {D}erivatives, {I}ntegrals,
  {S}eries and {O}ther {F}ormulas}.
\newblock Chapman and Hall/CRC, 2008.

\bibitem{Cvijovic}
D~Cvijovi{\'c}.
\newblock Closed-form summations of certain hypergeometric-type series
  containing the digamma function.
\newblock {\em Journal of Physics A: Mathematical and Theoretical},
  41(45):455205, 2008.

\bibitem{Doelder}
PJ~De~Doelder.
\newblock On some series containing $\psi(x)- \psi(y)$ and $(\psi(x)-
  \psi(y))^2$ for certain values of $x$ and $y$.
\newblock {\em Journal of Computational and Applied Mathematics},
  37(1-3):125--141, 1991.

\bibitem{DerivativesBesselJL}
JL~Gonz{\'a}lez-Santander.
\newblock Closed-form expressions for derivatives of {B}essel functions with
  respect to the order.
\newblock {\em Journal of Mathematical Analysis and Applications},
  466(1):1060--1081, 2018.

\bibitem{BetaReductionJL}
JL~Gonz{\'a}lez-Santander.
\newblock A note on some reduction formulas for the incomplete beta function
  and the {L}erch transcendent.
\newblock {\em Mathematics}, 9(13):1486, 2021.

\bibitem{SumsPsiJL}
JL~Gonz{\'a}lez-Santander and F~S{\'a}nchez~Lasheras.
\newblock Finite and infinite hypergeometric sums involving the digamma
  function.
\newblock {\em Mathematics}, 10(16):2990, 2022.

\bibitem{Hansen}
ER~Hansen.
\newblock {\em A {T}able of {S}eries and {P}roducts}.
\newblock Prentice-Hall, 1975.

\bibitem{Lebedev}
NN~Lebedev.
\newblock {\em Special {F}unctions and {T}heir {A}pplications}.
\newblock Prentice-Hall, Inc., 1965.

\bibitem{Miller}
AR~Miller.
\newblock Summations for certain series containing the digamma function.
\newblock {\em Journal of Physics A: Mathematical and General}, 39(12):3011,
  2006.

\bibitem{Atlas}
KB~Oldham, J~Myland, and J~Spanier.
\newblock {\em An {A}tlas of {F}unctions: {W}ith {E}quator, the {A}tlas
  {F}unction {C}alculator}.
\newblock Springer, 2009.

\bibitem{NIST}
Frank~WJ Olver, Daniel~W Lozier, RF~Boisvert, and ChW Clark.
\newblock {\em NIST {H}andbook of {M}athematical {F}unctions}.
\newblock Cambridge University Press, 2010.

\bibitem{Paris}
RB~Paris and D~Kaminski.
\newblock {\em Asymptotics and {M}ellin-{B}arnes integrals}, volume~85.
\newblock Cambridge University Press, 2001.

\bibitem{Prudnikov1}
AP~Prudnikov, YuA Brychkov, and OI~Marichev.
\newblock {\em Integrals and {S}eries: {M}ore {S}pecial {F}unctions}, volume~1.
\newblock CRC press, 1986.

\bibitem{Prudnikov3}
AP~Prudnikov, YuA Brychkov, and OI~Marichev.
\newblock {\em Integrals and {S}eries: {M}ore {S}pecial {F}unctions}, volume~3.
\newblock CRC press, 1986.

\end{thebibliography}

\end{document}